\documentclass[12pt,a4paper, reqno]{amsart}
\usepackage{amsmath,amsfonts,amssymb,mathrsfs}
\usepackage{mathtools}
\usepackage{comment}
\usepackage{xfrac}
\usepackage{subfigure}
\usepackage{bbm}
\usepackage[linktocpage=true,colorlinks=true,linkcolor=blue,citecolor=red,urlcolor=magenta,pdfborder={0 0 0}]{hyperref}
\usepackage{fancyhdr}
\usepackage{mathscinet}
\usepackage{dutchcal}
\usepackage{xcolor}
\usepackage[english]{babel}
\usepackage{mathscinet}

\usepackage{inputenc}[UTF-8]
\usepackage{ulem}

\definecolor{forestgreen}{rgb}{0.13, 0.55, 0.13}
\definecolor{anna}{rgb}{0.01, 0.28, 1.0}

\newtheorem{theorem}{\bf Theorem}[section]
\newtheorem{lemma}[theorem]{\bf Lemma}

\newtheorem{definition}[theorem]{\bf Definition}

\newtheorem{proposition}[theorem]{\bf Proposition}

\newcommand{\R}{\mathbb{R}}

\newcommand{\KK}{\mathbb{K}}

\newcommand{\W}{\mathcal{W}}

\def \dd {{\rm d}}

\def \KK {{\mathbb{K}}}

\def \L {\mathscr{L}}
\def \K {\mathscr{K}}

\def \A {\mathscr{A}}

\def \epsilon {{\varepsilon}}

\def \O {{\Omega}}

\def\p{\partial}
\def \tilde {\widetilde}

\usepackage{mathtools}

\usepackage{chngcntr}

\counterwithin*{equation}{section}
\counterwithin*{theorem}{section}

\usepackage{appendix}
\usepackage{apptools}
\AtAppendix{\counterwithin{theorem}{section}}

\begin{document}
	\title[]{On the obstacle problem associated to the Kolmogorov-Fokker-Planck operator
		  with rough coefficients}
	
	\author{Francesca Anceschi}
	\address{Dipartimento di Ingegneria Industriale e Scienze Matematiche -
		Università Politecnica delle Marche: via Brecce Bianche 12 - 60131 Ancona, Italy}
	\email{f.anceschi@staff.univpm.it}
	
	\author{Annalaura Rebucci}
	\address{Dipartimento di Scienze Fisiche, Informatiche e Matematiche - 
	Università degli Studi di Modena e Reggio Emilia: via Campi 213/b - 41115 Modena, Italy}
	\email{annalaura.rebucci@unipr.it}
	
	\thanks{\textbf{Aknowledgments:}
	The authors are supported by the INdAM - GNAMPA project 
	“Variational problems for Kolmogorov equations: long-time analysis and regularity estimates”, CUPE55F22000270001.}
	
	\date{\today}

	\begin{abstract}
		\noindent
		This work is devoted to the study of the obstacle problem associated to the Kolmogorov-Fokker-Planck operator
		  with rough coefficients through a variational approach. In particular, after the introduction of a proper anisotropic Sobolev space and related properties,
		 we prove the existence and uniqueness of a weak solution for the obstacle problem by adapting a classical perturbation argument to the convex functional
		associated to the case of our interest. 
		Finally, we conclude this work by providing a one-sided associated variational inequality, alongside with an overview on related 
		open problems.
				
		\medskip 
		\noindent
		{\bf Key words: Kolmogorov equation, weak regularity theory, ultraparabolic, obstacle problem}	
		
		\medskip
		\noindent	
		{\bf AMS subject classifications: 35Q84 - 35A15 - 47J20 - 35D30 -  35R03 }
	\end{abstract}
	
	\maketitle
	
	\hypersetup{bookmarksdepth=2}
	\setcounter{tocdepth}{1}
	
	\tableofcontents

	\section{Introduction}
	\label{intro}
	This work is devoted to the study of the obstacle problem associated to the following 
	Kolmogorov-Fokker-Planck operator
	\begin{align} \label{defL}
		\L u (v,x,t)  
		&= \nabla_v \cdot \left( A(v,x,t) \nabla_v u (v,x,t) \right) + v \cdot \nabla_x u(v,x,t) - \p_t u(v,x,t)\\ \nonumber
		&=: \nabla_v \cdot \left( A(v,x,t) \nabla_v u (v,x,t) \right)  + Y u(v,x,t),  
	\end{align}
	where $(v,x,t) \in \R^{2d+1}$ and the diffusion matrix $A$ satisfies the following assumption:
	\begin{itemize}
		\item[(H)] $A(v,x,t) = \left( a_{ij}(v,x,t) \right)_{ij}$ is a {bounded} $d \times d$ matrix made of real measurable coefficients and such that
				$a_{ij}=a_{ji}$ for every $i,j=1, \ldots, d$. Moreover, there exists two positive constants $\lambda$ and
				$\Lambda$ such that 
				\begin{equation}\label{unifell}
					\lambda |\xi|^2 \le  (A \xi ) \cdot \xi  \le \Lambda | \xi|^2 
					\qquad \forall \xi \in \R^d.
				\end{equation}
	\end{itemize} 
	Assumption (H) is usually referred to as ellipticity condition. In particular, the left-hand side of inequality \eqref{unifell}
	ensures coercivity of the functionals $J$ and $\mathscr{J}$ defined below in \eqref{functional} and \eqref{Jcors}, respectively, providing us with a fundamental
	ingredient for our proof. 
	Moreover, it is responsible of the degenerate nature of the Kolmogorov-Fokker-Planck operator $\L$, 
	for which the diffusion only happens in the velocity variables, i.e. the first $d$ directions. 

	\subsection{Motivation and background} The Kolmogorov-Fokker-Planck operator $\K$ with constant coefficient
	\begin{align*}
		\K u (v,x,t) = \Delta_v u(v,x,t) + v \cdot \nabla_x u(v,x,t) - \p_t u(v,x,t), \qquad (v,x,t) \in \R^{2d+1},
	\end{align*}
	was firstly introduced by Kolmogorov in \cite{K} as a fundamental ingredient for the study of the density of $d$ particles of gas 
	in the phase space. Later on, $\K$ was considered by H\"ormander in \cite{H} as a prototype for the family of hypoelliptic operators
	 of type $2$, i.e. the ones which can be written as a sum of squares plus a drift term 
	$\sum_{i=1}^{d} X_i^2 + X_0$, where $X_i$ is a smooth vector field for every $i =0, \ldots, d$, and for this reason we usually refer to 
	\begin{equation*}
		Yu(v,x,t) := v \cdot \nabla_x u(v,x,t) - \p_t u(v,x,t), \qquad (v,x,t) \in \R^{2d+1},
	\end{equation*} 
	as drift, or transport term. 
	This immediately suggested that $\K$ is a hypoelliptic operator because it satisfies H\"ormander's rank condition, meaning every solution to $\K u = f$ on a bounded open domain 
	$\O \subset \R^{2d+1}$ is smooth whenever $f \in C^{\infty}(\O)$. Thus, $\K$ possesses strong regularizing properties.

	As it will be clear in the following, operator $\L$ arises in various applications starting from the kinetic theory of gas, f. i. 
	\cite{Villani}, and the financial market modeling, f.i. \cite{PascucciBook, AMP}. On the other hand, from the purely analytical point of view, $\L$ is a 
	prototype for the family of second order ultraprabolic partial differential operators of Kolmogorov type defined as
	\begin{align} \label{ultra}
		\sum_{i,j=1}^{m_0}\partial_{x_i}\left(a_{ij}(x,t)\partial_{x_j}u(x,t)\right)+\sum_{i,j=1}^N b_{ij}x_j\partial_{x_i}u(x,t)-\partial_t u(x,t), 
		\qquad (x,t) \in \R^{N+1},
	\end{align}
	on Lie groups, see f. i. \cite{LP}. On one hand, when the matrix $A$ is made of constant, or H\"older continuous coefficients, 
	we deal with classical solutions. In this setting, Schauder estimates \cite{Manfredini, PRS}, well-posedness results for the Dirichlet \cite{Manfredini} and 
	the Cauchy problem \cite{PDF}, alongside with other results were proved by many authors over the years. For further information on this subject we refer 
	the reader to \cite{AP-survey} and the references therein. 
	
	On the other hand, when the matrix $A$ is made of measurable coefficients we have to tackle with 
	weak solutions. The extension of the De Giorgi-Nash-Moser regularity theory to this setting had been an open problem for decades recently resolved
	both in the kinetic and ultraparabolic setting, 
	with various contributions among which we recall \cite{GIMV, GI, GM} and \cite{AR-harnack}, respectively. 
	As far as we are concerned with well-posedness results for boundary value problems in the weak setting,
	there are recent results regarding existence and uniqueness of the solution for the Dirichlet problem \cite{LitsgardNystrom}, 
	existence of a weak fundamental solution for the weak Cauchy problem \cite{AR-funsol},
	and finally $C^{\alpha}$ regularity estimates up to the boundary \cite{Zhu, Silvestre}.
	It is in this weak framework that we aim to address the study of the weak obstacle problem by means of variational methods, a topic 
	which presents many interesting open problems that we discuss in Section \ref{var-ineq}.

	\medskip

Obstacle problems are not only fascinating for theoretical purposes, but also for their
applications in research areas as diverse as physics, biology and stochastic theory. 
We here focus on their connection to mathematical finance and in particular to the problem of determining the arbitrage-free price of American-style options. To be more precise, we consider a financial model where the time evolution of the state variables is described by the $2d$-dimensional diffusion process $X=\left(V_t^x,X_t^x \right)$ solving the stochastic differential equation
\begin{equation}\label{langevin}
\begin{cases}
    & d V_t^x=  \sqrt{2}\,d W_{t}, \\
	& d X_t^x = V_t^x	\, dt, 
\end{cases} \quad \textit{{\rm and}} \quad X_{t_0}^{t_0,x}=x,
\end{equation}
where $(x,t_0)\in \R^{2d}\times [0,T]$ and $W_t$ denotes a $d$-dimensional Wiener process. 

We recall that an American option with pay-off $\psi$ is a contract which grants the holder the right to receive the payment of the sum $\psi(X_t)$ at a time $t \in [0,T]$, which is chosen by the holder. Then, in virtue of the classical arbitrage theory (see, for instance {\cite{PascucciBook}}), the fair price at the initial time $0$ of the American option, assuming the risk-free interest rate is zero, is given by the following optimal stopping problem
\begin{equation}\label{optimal}
u(x,t)=\sup_{\tau \in [t,T]}E \left[ \psi \left( X_\tau^{t,x} \right) \right],
\end{equation}
where the supremum is taken over all stopping times $\tau \in [t,T]$ of $X$. In \cite{Pascucci}, it is proved that the function $u$ in \eqref{optimal} is a solution to the obstacle problem
\begin{align}
		\label{obstaclePP}
		\begin{cases}
			\max \lbrace  \L u -f, \psi -u\rbrace \qquad &(v,x,t) \in \R^{2d} \times [0,T] \\
			u(v,x,t) = g  \qquad &(v,x,t) \in \R^{2d} \times \lbrace 0 \rbrace,
		\end{cases}
	\end{align}
where the obstacle $\psi$ corresponds to the pay-off of the option and it is a Lipschitz continuous function in $\overline{\O}$ satisfying the following condition: there exists a constant $c \in \R$ such that
\begin{equation}\label{cond-obstacle}
\sum_{i,j=1}^d \xi_i \xi_j \partial_{x_i x_j}\psi \geq c|\xi|^2, \quad \textit{in $\O$}, \xi \in \R^d
\end{equation}
in the distributional sense. 

In the uniformly parabolic case, i.e. when $d=N$, the evaluation of American options was studied starting from the article \cite{Ben}, where a probabilistic approach was employed, and later on developed in the paper \cite{Jailler} considering a variational approach.

Furthermore, the investigation of problem \eqref{obstaclePP} is motivated by the fact that there are significant classes of American options whose corresponding diffusion process $X$ is associated to Kolmogorov-type operators which are not uniformly parabolic and are of the kind \eqref{defL}. Two such examples are given by American style options (c.f. \cite{Barucci}) and by the American options priced in the stochastic volatility introduced in the article \cite{Hobson}.

In virtue of its importance in finance, the mathematical study of the obstacle problem \eqref{obstaclePP} was already initiated in the papers \cite{obstacle-ex, obstacle-2, obstacle-3}. Specifically, the main result of \cite{obstacle-ex} is the existence of a strong solution to problem \eqref{obstaclePP} in certain bounded cylindrical domains and in the strips $\R^{2d} \times [0,T]$ through the adaptation of a classical penalization technique. On the other hand, the main purpose of papers \cite{obstacle-2, obstacle-3} is to prove some new regularity results for solutions to \eqref{obstaclePP}. In particular, \cite{obstacle-2} concerns the optimal interior regularity for solutions to the problem \eqref{obstaclePP}, while \cite{obstacle-3} contains new results regarding the regularity near the initial state for solutions to the Cauchy-Dirichlet problem and to \eqref{obstaclePP}. 

However, in the aforementioned papers \cite{obstacle-ex, obstacle-2, obstacle-3}, the authors could only deal with strong solutions and continuous obstacles satisfying condition \eqref{cond-obstacle}. For this reason, the main purpose of this paper is to improve the results contained in \cite{obstacle-ex, obstacle-2, obstacle-3, Pascucci} by studying the obstacle problem \eqref{obstaclePP} in a more general and natural setting, i.e. by considering weak solutions to \eqref{defL} in the functional space $\W$. Furthermore, in a standard manner (see 
	\cite[Chapter 6]{KinderlehrerStampacchia}), we assume that the obstacle $\psi$ and the boundary data $g$ inherit the same regularity of the function $u$, namely $\psi \in \W(\O_v \times \O_{xt})$ and $g \in \W(\O_v \times \O_{xt})$. In comparison with \cite{obstacle-ex, obstacle-2, obstacle-3, Pascucci}, we also weaken the regularity assumptions on the right-hand side by considering  $f \in L^2(\O_{xt}, H^{-1}(\O_v))$ and by considering the following more general obstacle problem
	\begin{align}
		\label{obstacle}
		\begin{cases}
			\L u(v,x,t) = f(v,x,t) \qquad &(v,x,t) \in \O \\
			u(v,x,t) \ge \psi(v,x,t)  \qquad &(v,x,t) \in \O \\
			u(v,x,t) = g  \qquad &(v,x,t) \in \p_K \O,
		\end{cases}
	\end{align}
	where the boundary condition needs to be considered as attained in the sense of traces, the obstacle condition holds in $\W(\O_v \times \O_{xt})$
	and 
	\begin{equation} \label{ordering-ob}
		\psi \le g \, \, \, \text{on } \p_K (\O_{v} \times \O_{xt})
		\quad \text{in } \, \, \W(\O_v \times \O_{xt}),
	\end{equation}
an ordering relation whose meaning will be clarified in Section \ref{preliminaries}.
Another motivation behind our studies is to pursue the variational analysis of Kolmogorov-type operators. In particular, with this work we aim at initiating the study of the obstacle problem \eqref{obstacle} in the framework of Calculus of Variations, in order to take advantage of the rich toolbox provided by such theory when it comes to investigating weak solutions and less regular obstacles. 

A first step towards this direction was already taken in \cite{AAMN, AR-harnack, LitsgardNystrom}, where the natural functional setting for the variational study of degenerate Kolmogorov equations was identified and subsequently characterized. This leads back to finding the right variational formulation and the right functional associated to Kolmogorov-type equations. More precisely, following \cite{AAMN}, we rewrite the problem of finding a solution to \eqref{obstacle} as that of finding a null minimizer of the functional
\begin{align}
		\label{functional-intro}
		 \inf \left\{  \,\, \iiint \limits_{\O_v \times \O_{xt}} \frac12 
		\left( A \left( \nabla_v u - \mathbb{j} \right) \right) \cdot
					\left( \nabla_v u - \mathbb{j} \right) \, \dd v \, \dd x \, \dd t \, : \, \nabla_v \cdot \left( A \mathbb{j} \right) = f - Yu \right\}.
	\end{align}
It is clear that the infimum in \eqref{functional-intro} is non-negative	and that, given a solution $u$ to \eqref{obstacle}, if we choose $\mathbb{j}=\nabla_v u$, then \eqref{functional-intro} vanishes at $u$. Moreover, it is easy to show that the functional in \eqref{functional-intro} is uniformly convex and attains its minimum at zero. Finally, we remark that the functional justifies the definition of functional kinetic space given in \eqref{kin-space}.

\subsection{Functional setting}
In order to present our main results, we first need to introduce some further notation regarding both the domain and
	the functional setting we are considering from now on in this work. First of all, let $\O := \O_v \times \O_{xt}$ be a subset of $\R^{2d+1}$ such that 
	\begin{itemize}
	\item[(D)] $\O_v \subset \R^d$ is a bounded Lipschitz domain 
	and $\O_{xt} \subset \R^{d+1}$ is a bounded domain with 
	$C^{1,1}$-boundary, i.e. $C^{1,1}$-smooth with respect to the transport operator $Y$ as well as $t$.
	\end{itemize}
	Then, if we denote by $N$ the outer normal to $\O_{xt}$, we are able to classically define the 
	Kolmogorov boundary of the set $\Omega$ as
	\begin{equation} \label{boundary-k}
		\p_K ( \O_v  \times \O_{xt} ) := 
		\left( \p \O_{v} \times \O_{xt} \right) \cup 
		\left\{ (v,x,t) \in \overline \O_v \times \p \O_{xt} \, | \, \, 
		(v, -1) \cdot N_{xt} > 0 \right\},
	\end{equation}
	which serves in the context of the operator $\L$ as the natural hypoelliptic counterpart of the parabolic boundary considered in the context of Cauchy-Dirichlet problems for uniformly parabolic equations. 
	Notice that the Kolmogorov boundary is well defined on the domain $\O_v  \times \O_{xt} $,
	since we assume enough regularity on the boundary of $\O_{xt}$ to ensure the existence of the
	normal $N_{xt}$. 

	Secondly, let
	us denote by $H^1(\O_v)$ the Sobolev space in the velocity variable, i.e. the space of 
	functions whose distributional gradient in $\O_v$ lies in $\left( L^2 (\O_v) \right)^d$ that is
	\begin{align}\label{kin-space}
		H^1(\O_v) := \left\{ h \in L^2(\O_v) : \, \, \nabla_v h \in \left( L^2 (\O_v) \right)^d 
		\right\}.
	\end{align}
	Moreover, we set the norm associated to the space $H^1(\O_v)$ as
	\begin{align*}
		\| h \|_{H^1 (\O_v) } := \| h \|_{L^2(\O_v)} + \| \nabla_v h \|_{L^2(\O_v)},
		\qquad \forall h \in H^1(\O_v),
	\end{align*}	
	where by abuse of notation $ \| \nabla_v h \|_{L^2(\O_v)}$ is the vectorial norm of the gradient. 
	Then, following the classical approach, we let $H^1_c(\O_v)$ denote the closure of 
	$C_c^\infty (\O_v)$ in the norm of $H^1(\O_v)$ and we recall that $C^\infty_c(\overline \O_v)$
	is dense in $H^1(\O_v)$ given that $\O_v$ is a Lipschitz domain by assumption. 
	This means we can define equivalently $H^1(\O_v)$ as the closure of $C^\infty_c (\overline 
	\O_v)$ in the norm $\| \cdot \|_{H^1(\O_v)}$. 	
	Since $H^1_c(\O_v)$ is a Hilbert space, then it is reflexive, meaning 
	$$
		\left( H^1_c (\O_v) \right)^* = H^{-1} (\O_v) 
		\quad \text{and} \quad
		\left( H^{-1} (\O_v) \right)^* = H^{1}_c (\O_v) ,
	$$
	where $( \cdot)^*$ denotes the dual of the space. Hence, from now on we denote by $H^{-1}
	(\O_v)$ the dual to $H^1_c(\O_v)$ acting on functions in $H^1_c(\O_v)$ through the 
	duality pairing 
	$$\langle \, \cdot \, | \, \cdot \, \rangle := 
	\langle \, \cdot \, | \, \cdot \, \rangle_{H^{-1}(\O_v), H^{1}_c(\O_v)}.
$$
	Now, following the approach proposed in \cite{LitsgardNystrom}, we define the Kolmogorov Sobolev space
	$\W(\O_v \times \O_{xt})$ as the closure of $C^{\infty}(\overline{\O_{xt} \times \O_{v}})$ in the norm
	\begin{align*}
		\| w \|_{\W( \O_{xt} \times \O_{v})} 
		&:=
		\| w \|_{L^2( \O_{xt}, H^1(\O_{v}))} + 
		\| Y u \|_{L^2(\O_{xt},H^{-1}(\O_v)} \\
		&:= \left( \iint_{\O_{xt}} \| w ( \cdot, x,t) \|_{H^1(\O_v)}^2 \dd x \, \dd t \right)^{\frac12} 
		+ \left( \iint \limits_{\O_{xt}} \| Y u (\cdot, x, t) \|^2_{H^{-1}(\O_v)} \dd x \, \dd t \right)^{\frac12}. 
	\end{align*}
We recall that in \cite{AAMN} the authors prove that $C^\infty_c(\overline{\O_{xt} \times \O_v})$ is 
	dense in the space 
	\begin{align*}
		\W(\Omega_{xt} \times \O_v) = 
		\left\{ w \in L^2 (\O_{xt}, H^1 (\O_v) ): \, Y w \in L^2(\O_{xt},H^{-1}(\O_v)) \right\}.
	\end{align*}
	Hence, we could also consider this equivalent definition for the Kolmogorov Sobolev space.
	Lastly, we consider that $L^2(\O_{xt}, H^{-1}(\O_v))$ is such that	
	\begin{align*}
		L^2(\O_{xt}, H^{-1}(\O_v)) = \left( L^2(\O_{xt}, H^{1}_c(\O_v))  \right)^*.
	\end{align*}
Furthermore, if we consider again $\O_{xt} \times \O_v$ under the assumption (D), then we denote by 
	$\p ( \O_{xt} \times \O_v )$ its topological boundary and by $\p_K ( \O_{xt} \times \O_v )$ its 
	Kolmogorov boundary as defined in \eqref{boundary-k}, where $\p_K ( \O_{xt} \times \O_v )
	\subset \p ( \O_{xt} \times \O_v )$. Then we denote by 
	$C_{0,K}^{\infty} (\overline{ \O_{xt} \times \O_v })$ the set of functions 
	$C^{\infty} (\overline{ \O_{xt} \times \O_v })$ vanishing on $\p_K (\O_{xt} \times \O_v )$.
	Additionally, we denote by $\W_0(\O_{xt} \times \O_v )$ the closure in the norm 
	$\W(\O_{xt} \times \O_v )$ of $C_{0,K}^{\infty} (\overline{ \O_{xt} \times \O_v })$. 
	
We conclude this subsection recalling the following Poincar\`{e} inequality for functions in $L^2( \O_{xt}, H^1_{c}(\O_{v}))$ (see \cite[Lemma 2.2]{LitsgardNystrom}).
\begin{lemma}\label{lemmapoinc}
There exists a constant $1 \leq C<+\infty$, which depends only on $d$ and $\O_v$, such that 
\begin{equation*}
\Vert u \Vert_{L^2( \O_{xt}, L^2(\O_{v}))} \leq C \Vert\nabla_v u \Vert_{L^2( \O_{xt}, L^2(\O_{v}))}
\end{equation*}
for every $u \in L^2( \O_{xt}, H^1_{c}(\O_{v}))$.
\end{lemma}

	\subsection{Main results}
As previously pointed out, our aim is to study the well-posedness theory for weak solutions {to the obstacle problem \eqref{obstacle}}
	under assumption (H) on a domain $\Omega$ satisfying assumption (D).
	Given the notation introduced in the previous subsection, we are in a position to properly formalize the definition of weak solution
	to \eqref{obstacle} considered in this work. 
	\begin{definition}
		\label{weak-sol}
		We say $u$ is a weak solution to \eqref{obstacle} if
		\begin{equation*}
			u \in \W (\O_{v} \times \O_{xt}) , \quad 
			u \in g + \W_0(\O_{v} \times \O_{xt}),
			\quad u \ge \psi \, \, \, \text{in } \,\, \W(\O_v \times \O_{xt})
		\end{equation*}
		and such that 
		\begin{align}\label{weakformu}
			0 = 
			 \iiint \limits_{\O_{v} \times \O_{xt}} A(v,x,t) \nabla_v u \cdot \nabla_v \phi 
				\, \dd v \, \dd x \, \dd t +
			     \iint \limits_{\O_{xt}} \langle f (\cdot, x,t) - Y u (\cdot, x,t) | \phi(\cdot, x,t) \rangle 
			     \, \dd x \, \dd t     
		\end{align}
		for every $\phi \in L^2(\O_{xt}, H^1_c(\O_v))$ and where $\langle \cdot | \cdot \rangle$ is the standard duality pairing
		in $H^{-1}(\O_v)$.	
	\end{definition}
	As it also happens in the parabolic setting, a weak solution to $\L u = 0$ in the sense of 
	the above definition 
	is also a weak solution in the sense of distributions. Indeed, whenever $\phi \in C_c^{\infty}(\O_{v} 
	\times \O_{xt})$
	we have 
	\begin{equation*}
		0 = \iiint \limits_{\O_{v} \times \O_{xt} }
			\left(  A(v,x,t) \nabla_v u \cdot \nabla_v \phi 
				+ u Y \phi \right)
			     \, \dd v \, \dd x \, \dd t     . 
	\end{equation*}
		
We are now in a position to state our main result.
	\begin{theorem}
		\label{main-thm}
		Let assumptions (H) and (D) hold. Let $f \in L^2(\O_{xt}, H^{-1}(\O_v))$ and
		$g, \psi \in \W(\O_v \times \O_{xt})$.
		Then there exists a unique weak solution $u \in \W(\O_{v} \times \O_{xt})$
		in the sense of Definition \ref{weak-sol}
		to the obstacle problem \eqref{obstacle}. 
		Moreover, there exists a constant $C$, which only depends on $d$ and on $\O_{v} \times \O_{xt}$, such that
		\begin{equation}\label{quantest}
		\| u \|_{\W( \O_{xt} \times \O_{v})} \leq C 
		\left( \| g \|_{\W( \O_{xt} \times \O_{v})}+\| f \|_{L^2(\O_{xt}, H^{-1}(\O_v))}\right).
		\end{equation}				
	\end{theorem}
	
\subsection{Outline of the paper}
This work is organized as follows. In Section \ref{preliminaries} we discuss the non-Euclidean geometrical setting suitable for 
		the study of operator $\L$, then we provide the reader with a characterization of non-negative functions and of the maximum 
		between two functions in the space $\W$. Section \ref{main-proof} is devoted to the proof of our main result, Theorem \ref{main-thm}. 
		Finally, Section \ref{var-ineq} concludes this work and contains the proof of a one-sided associated variational inequality, alongside with an overview on related 
		open problems.

	\section{Preliminaries}
	\label{preliminaries}

	\subsection{Geometrical setting}
	As firstly pointed out by Lanconelli and Polidoro in \cite{LP}, a Lie group structure is the most suitable geometrical setting
	for studying operator $\L$. Hence, we endow $\R^{2d+1}$ with the non- commutative group law
	\begin{equation*}
		z_0 \circ z = (v_0+v,x_0+x+tv_0,t_0+t), \qquad \forall z_0=(v_0,x_0,t_0) \in \R^{2d+1},
	\end{equation*}
	also known as Galilean change of variables. Then $\mathbb{G}:=(\R^{2d+1}, \circ)$ is a Lie group with identity 			
	element~$e:=(0,0,0)$ and inverse defined by
	$$
		z^{-1}:=(-v,-x+tv,-t), \qquad \qquad  \quad \forall z=(v,x,t) \in \R^{2d+1}.
	$$
	We observe that $\K$ is left invariant with respect to the group operation $\circ$.  Specifically, if $w(v,x,t) = u (v_{0} + v, x_{0} + x + t v_{0}, t_{0} + t)$ and $g(v,x,t) = f(v_{0} + v, x_{0} + x + t v_{0}, t_{0} + t)$, then 
\begin{equation*}
	\K u = f \quad \iff \quad \K w = g \quad \text{for  every} \quad (v_{0}, x_{0}, t_{0}) \in \R^{2m+1}.
\end{equation*}
Moreover,
	we are also allowed to introduce a family of dilations 
	\begin{align*}
		\delta_r(z) = (r v, r^3 x,r^2 t), \quad \quad \forall r >0, \, \forall z=(v,x,t) \in \R^{2d+1}.
	\end{align*}
	with respect to which the operator $\K$ is invariant. Indeed,
	if~$u$ is a solution to $\K u = 0$, then for every~$r>0$ the scaled function $u_r(z) = u(\delta_r(z))$
	satisfies the same equation in a suitably rescaled domain. Hence, we say $\K $ 
	is homogeneous of degree~$2$ with respect to the dilation group~$\{\delta_{r}\}_{r>0}$.

	\subsection{Well-posedness of the obstacle problem}
	 First of all, by paralleling the definition of a non-negative function in $H^1$ 
	\cite[Definition 5.1]{KinderlehrerStampacchia} we introduce the definition of non-negative function
	in the sense of $\W(\O_v \times \O_{xt})$.
		
	\begin{definition} \label{non-neg}
		Let $w \in \W(\O_v \times \O_{xt})$ and $E \subset \overline{\O_{v} \times \O_{xt}}$
		The function $w$ is nonnegative on $E$ in the sense of $\W(\O_v \times \O_{xt})$,
		if there exists a sequence $w_n \in C^{\infty}_{\!\!c}(\O_{v} \times \O_{xt})$ such that 
		\begin{align*}
			w_n \ge 0 \, \, \text{in } E \quad \text{and} \quad w_n \to w \, \, \text{in } \W(\O_{v} \times \O_{xt}).
		\end{align*}
		If $-w \ge 0$ on $E$ in the sense of $\W(\O_v \times \O_{xt})$, then $w$ is said to be nonpositive 
		on $E$ in the sense of $\W(\O_v \times \O_{xt})$. If $w$ is both $w\ge0$
		and $w \le 0$ on $E$ in the sense of $\W(\O_v \times \O_{xt})$, then $w=0$ on $E$ in the sense of $\W(\O_v \times \O_{xt})$.
	\end{definition}
	Similarly, for any $w,z \in \W(\O_v \times \O_{xt})$ we say $w \le z$ on $E$ in the sense of $\W(\O_v \times \O_{xt})$ if 
	$w -z \ge 0$ on $E$ in the sense of $\W(\O_v \times \O_{xt})$.
	In particular, $z$ may be constant. Hence, we are able to introduce the definition 
	\begin{equation*}
		\sup \limits_E w = \inf \, \left\{ M \in \R: \, \, w \le M \, \, \text{on $E$ in the sense of $\W(\O_v \times \O_{xt})$} \right\},
	\end{equation*}
	and analogously of 
	\begin{equation*} 
		\inf \limits_E w = \sup \left\{ m \in \R: \, \, w \ge m \, \, \text{on $E$ in the sense of $\W(\O_v \times \O_{xt})$} \right\}.
	\end{equation*}
	
	Now, we prove an 
	explicit characterization for non-negative functions on $E$ in the sense of $\W(\O_v \times \O_{xt})$. 
	This result plays an important role in our analysis, since it allows us to define the set $\KK_{\psi}$ in \eqref{K-def},
	which is the starting point of our variational analysis for the obstacle problem.
	\begin{lemma}
		\label{K-equiv}
		Let $w \in \W(\O_v \times \O_{xt})$ and $E \subset \overline{\O_{v} \times \O_{xt}}$ be bounded. 
		Then $w \ge 0$ on $E$ in the sense of $\W(\O_v \times \O_{xt})$, if and only if
		$w \ge 0$ on $E$ a.e. 
	\end{lemma}
	
	\begin{proof} The right implication trivially follows by considering that $w \ge 0$ 
		in the sense of $\W (\O_{v} \times \O_{xt})$ implies $w \ge 0$  also in the sense of $L^2(\O_{v},L^2( \O_{xt}))$; 
		and hence a.e.
				
		As far as we are concerned with the left implication, given our definition of $\W(\O_v 
		\times \O_{xt})$, there exists a sequence of functions $w_n \in C^\infty_{0,K}(\O_v 
		\times \O_{xt})$ such that $w_n \to w$ in $\W(\O_v \times \O_{xt})$ and in
		$\O_v \times \O_{xt}$ pointwise a.e. Hence, $\max ( w_n, 0) \ge 0$ and 
		$w = \max (w, 0)$ in $\O_v \times \O_{xt}$ a.e. So, we have
		\begin{align*}
			\| \max ( w_n, 0) - w \|_{L^2(\O_{v},L^2( \O_{xt}))} &= \| \max ( w_n, 0) -  \max (w, 0) \|_{L^2(\O_{v},L^2( \O_{xt}))} \\
			&\le \| w_n - w \|_{L^2(\O_{v},L^2( \O_{xt}))}  \to 0 \quad \text{in } L^2(\O_{v},L^2( \O_{xt}))
			\, \, \text{as } n \to \infty.
		\end{align*}
		From this, it follows
		\begin{align*}
			\iiint \limits_{\O_v \times \O_{xt}} \max (w_n , 0)^2 \, \dd v \, \dd x \, \dd t
			\le \iiint \limits_{\O_v \times \O_{xt}} w_n^2 \, \dd v \, \dd x \, \dd t \le C .
		\end{align*}
		Hence, the sequence $\max (w_n, 0)$ has a subsequence that converges weakly in 
		$\W(\O_v \times \O_{xt})$ to a certain element, that needs to be $w$ considering the previous 
		computations. Hence, by Mazur's lemma,	we conclude the proof.	
	\end{proof}

	Finally, taking into consideration the results above and the H\"older regularity results for weak solution to $\L u = f$ 
	proved in \cite{AR-harnack, GI, GM, Zhu, Silvestre}
	(see Section \ref{intro}), the ordering relation between the obstacle and the boundary data
	introduced in \eqref{ordering-ob} is well-posed and classically justified as in \cite[Chapter 10]{Brezis}
	by applying the classical maximum principle proved in \cite{Bony}. Hence, 
	the obstacle problem \eqref{obstacle} is well-posed.

\section{Proof of Theorem \ref{main-thm}}
\label{main-proof}
This section is devoted to the proof of Theorem \ref{main-thm} via an adaptation of the method proposed in \cite{LitsgardNystrom}
to address the study of a Dirichlet problem. 

	Throghout this section, we consider a fixed obstacle $\psi \in \W(\O_v \times \O_{xt})$, right-hand side $f \in L^2(\O_{xt}, H^{-1}(\O_v))$ and 
	boundary data $g \in \W(\O_v \times \O_{xt})$. Then, we introduce 
	the set:
	\begin{align} \label{K-def}
		\KK(\psi,g) := \left\{ w \in  \W(\O_{v} \times \O_{xt}) : \, \,  w \in g + \W_0(\O_{v} \times \O_{xt}), \right. \\ \nonumber 
		\left. w \ge \psi \, \, \text{in  } \O  \,\, \text{in } \, \W(\O_v \times \O_{xt}) \right\}.
	\end{align}
	Then we note it is not empty, as $\psi \in \KK(\psi,g)$. 
	Moreover, $\KK(\psi,g)$ is a convex set and, thanks to Lemma \ref{K-equiv}, $w \in \KK(\psi,g)$ if and only if
	\begin{align*}
		w \in  \W(\O_{v} \times \O_{xt}), \, \,  w \in g + \W_0(\O_{v} \times \O_{xt}), \quad \text{and} \quad w \ge \psi \, \, \text{a.e. in } \, 
		 \W(\O_{v} \times \O_{xt}).
	\end{align*} 
	Now, for every pair of functions $(w,f)$ such that
	\begin{equation*}
		w \in L^2(\O_{xt}, H^1(\O_v)) \quad \text{and} \quad f - Y w \in 
		L^2 (\O_{xt}, H^{-1}(\O_v)),
	\end{equation*}
	we may introduce a functional $J$ defined as
	\begin{align}
		\label{functional}
		J [w, f] = \, \inf \limits_ {G(w,f)} \,  \, \iiint \limits_{\O_v \times \O_{xt}} \frac12 
		\left( A \left( \nabla_v w - \mathbb{j} \right) \right) \cdot
					\left( \nabla_v w - \mathbb{j} \right) \, \dd v \, \dd x \, \dd t
	\end{align}
	where the infimum is taken over the set
	\begin{align}
		\label{Gw}
		G(w,f) = \left\{ \mathbb{j} \in (L^2(\O_{xt}, L^2(\O_v))^d: \, \nabla_v \cdot \left( A \mathbb{j} \right) = f - Yw \right\}.
	\end{align}
	The functional $J$ is well-defined, since we are interested in solutions to \eqref{obstacle}, where 
	$f \in L^2 (\O_{xt}, H^{-1}(\O_v))$ and $\psi,g \in \W(\O_{v} \times \O_{xt})$ are fixed by 
	the problem.
	Moreover, we highlight that the condition each $\mathbb{j}$ in $G(w,f)$ needs to satisfy 
	is intended in the sense of distributions, i.e. 
		\begin{align*}
			- \iiint \limits_{\O_{v} \times \O_{xt}} A \mathbb{j} \cdot \nabla_v \phi \, \dd v \, \dd x \, \dd t 
			= \iint \limits_{\O_{xt}} \langle (f - Y w)(\cdot,x,t) | \phi(\cdot,x,t) \rangle  \dd x \, \dd t 
			\qquad \forall \phi \in L^2 \left( \O_{xt}, H^1(\O_v) \right),
		\end{align*}
	where $\langle \cdot | \cdot \rangle$ is the standard duality pairing in $H^{-1}(\O_v)$.
	
	Then, we are in a position to prove the one-to-one correspondence between the solution of the 
	obstacle problem and the 
	(zero) minimizer of the functional $J$. 
	Note that an analogous equivalence result for the Dirichlet problem was proved in
	\cite[Lemma 3.3]{LitsgardNystrom}.  
	\begin{lemma}
		\label{equivalenza1}
		Let assumptions (H) and (D) hold. Let $f \in L^2(\O_{xt}, H^1(\O_v))$, $\psi,g \in \W(\O_v \times \O_{xt})$ be fixed by the problem 
		\eqref{obstacle}. Then $u \in \KK(\psi,g)$ is a solution to \eqref{obstacle} if and only if 
		\begin{equation*}
			0 = J[u,f] = \min \limits_{w \in \KK(\psi,g)} J[w,f].
		\end{equation*}
	\end{lemma}
	\begin{proof}
		The proof of this equivalence result relies on the observation that, since $A$ satisfies the 
		ellipticity assumption (H), then $J[w,f] \ge 0$ for every $w \in \KK(\psi,g)$ and therefore we need to prove that
\begin{equation*}
u \textit{ {\rm is a weak solution to \eqref{obstacle}}}\quad \iff \quad J[u,f]=0.
\end{equation*}	
		\noindent
		$(\Longrightarrow)$ Let $u \in \W(\O_v \times \O_{xt})$ be a solution to \eqref{obstacle}, then our aim is to show 
		$J[u,f]=0$. Notice that by definition \eqref{functional} we have
		\begin{align*}
			J[u, f] = \, \inf \limits_{G(u,f)} \, \, \iiint \limits_{\O_v \times \O_{xt}}
			\frac12 \left( A \left( \nabla_v u -\mathbb{j} \right) \right)\cdot \left( \nabla_v u - \mathbb{j} \right)
			\, \dd v \, \dd x \, \dd t,
		\end{align*}
		where $G(u,f)$ is defined as in \eqref{Gw}. Given that $u$ is a solution to \eqref{obstacle}
		by assumption, then the infimum in \eqref{functional} is attained at $\mathbb{j} = \nabla_v u$.
		Hence, $J[u, f] = 0$.
		
		\medskip
		
		\noindent
		$(\Longleftarrow)$ Let $u \in \KK(\psi,g)$ be such that $J[u, f]=0$, then our aim is to
		show $u$ is a solution to \eqref{obstacle}. Given our assumptions, we have
		\begin{align*}
			0 = J [u, f] = \, \inf \limits_ {G(u,f)} \, \, \iiint \limits_{\O_v \times \O_{xt}} 
			\frac12 \left( A \left( \nabla_v u - \mathbb{j} \right) \right) \cdot
					\left( \nabla_v u - \mathbb{j} \right) \, \dd v \, \dd x \, \dd t,
		\end{align*}
		hence $\mathbb{j} = \nabla_v u$ a.e. in $\O_v \times \O_{xt}$, as the integrand is non-negative. Finally, given the definition of $\mathbb{j}$ in \eqref{Gw}, 
		we conclude 
		\begin{equation*}
			\iiint \limits_{\O_v \times \O_{xt}} A \nabla_v u \cdot \nabla_v \phi \, \dd v \, \dd x \, \dd t
			= \iiint \limits_{\O_v \times \O_{xt}} f \phi \, \dd v \, \dd x \, \dd t - \iint \limits_{ \O_{xt}}
			\langle Y u(\cdot, x,t) | \phi(\cdot, x,t) \rangle \, \, \dd x \, \dd t
		\end{equation*}
		for every $\phi \in L^2 (\O_{xt}, H^1(\O_v))$.
		Hence, $u$ is by definition a solution to \eqref{obstacle}.
	\end{proof}
	Hence, in order to prove Theorem \ref{main-thm} we reduce ourselves 
	to show the existence of a (zero) minimizer for $J$. 
	
	\medskip
	
	Now, let us consider problem \eqref{obstacle} and, given the data $f$, the obstacle $\psi$ and the boundary data $g$, 
	we introduce the set of pairs of functions satisfying it:
	\begin{equation} \label{Acors}
			\A(f,\psi, g) =  \left\{ (u,\mathbb{j}) \in \KK(\psi,g) \times (L^2(\O_{xt}, L^2(\O_v))^d \, | \, \nabla_v \cdot
			(A(v,x,t) \mathbb{j}) = f - Yu \right\}.
		\end{equation}
	\begin{lemma}
		\label{not-empty}
		Let $(f, \psi, g) \in L^2(\O_{xt}, H^{-1}(\O_v))\times \W(\O_v \times \O_{xt}) \times \W(\O_v \times \O_{xt})$ be fixed. 
		Then the set $\A(f, \psi, g)$ is not empty. 
	\end{lemma}
	\begin{proof}
		To prove this result, it is sufficient to show there exists at least a pair of functions belonging to $\A(f,\psi,g)$.
		Our idea is to apply the Lax-Milgram Theorem considering the Hilbert space $H^1_c(\O_v)$ 
		and the set $H^{-1}_c(\O_v)$ of linear functionals 
		on $H^1_c(\O_v)$. Hence, for every $w,\phi \in H^1_c(\O_v)$ we introduce the bilinear form
		$$
			b (w,\phi) = \int \limits_{\O_v}  A \nabla_v w \cdot \nabla_v \phi \, \dd v .
		$$
We observe that, in virtue of the boundedness of $A$ and the Cauchy-Schwarz inequality, the bilinear form is continuous, i.e.
\begin{equation*}
\left\vert b(w,\phi)\right\vert \leq |A|  \| \nabla_v w \|^{2}_{L^2(\O_v)}  \| \nabla_v \phi \|^{2}_{L^2(\O_v)}. 
\end{equation*}
		Now, we need to show the bilinear form $b$ is coercive in $H^1_c(\O_v)$. 
		By the ellipticity assumption (H) for the matrix $A$, we have
		\begin{align*}
			b (w,w) = \int \limits_{\O_v}  A \nabla_v w \cdot \nabla_v w \, \dd v \ge \lambda \| \nabla_v w \|^{2}_{L^2(\O_v)}.
		\end{align*}
		Then, by the Lax-Milgram theorem, for any bounded functional $L \in H^{-1}_c(\O_v)$ there exists a unique 
		$w$ such that 
		$$
			L(\phi) = b(w, \phi) \quad \text{in } H^{-1}_c(\O_v). 
		$$
		Thus, by choosing $\phi=g$ and $L=  f - Yu$, for a.e. fixed $(x,t) \in \O_{xt}$ the equation
		\begin{equation*}
			 \nabla_v \cdot (A(v,x,t) \nabla_v w(v,x,t)) = f(v,x,t) - Yu(v,x,t) 
		\end{equation*}
		admits a unique solution $w(\cdot)= w (\cdot, x,t) \in H^1_c(\O_v)$. In particular, the set $\A$ is not empty, since
		at least we have 
		$$
			(g, \nabla_v w) \in \A(f,\psi,g).
		$$

	\end{proof}
	
	\medskip
	
For every $u \in\KK(\psi,g)$ and $\mathbb{j} \in (L^2(\O_{xt}, L^2(\O_v))^d$, we now define the functional 
\begin{equation}\label{Jcors}
\mathscr{J}[u,\mathbb{j}]:=\iiint \limits_{\O_v \times \O_{xt}}
			\frac12 \left( A \left( \nabla_v u -\mathbb{j} \right) \right)\cdot \left( \nabla_v u - \mathbb{j} \right)
			\, \dd v \, \dd x \, \dd t,
\end{equation}
which is uniformly convex on $\A(f,\psi,g)$, as proved in \cite[Lemma 3.2]{LitsgardNystrom}. 
As a consequence, there exists a unique minimizing pair $(\tilde{u},\tilde{\mathbb{j}}) \in \A(f,\psi,g)$ such that
\begin{equation}\label{argmin}
	(\tilde{u},\tilde{\mathbb{j}}) = \arg\min \limits_{(u,\mathbb{j})\in \A(f,\psi,g)} \, \mathscr{J}[u,\mathbb{j}].
\end{equation}
We observe that
\begin{equation*}
	\min_{(u,\mathbb{j}) \in \A(f,\psi,g)}\mathscr{J}[u,\mathbb{j}]=\min_{u \in \KK(\psi,g)}J[u,f].
\end{equation*}
Hence, by construction and considering the ellipticity of $A$, we infer $J[\tilde{u},f] \geq 0$. 
Thus, keeping in mind Lemma \ref{equivalenza1}, it is sufficient to show that 
\begin{equation}\label{negfunc}
J[\tilde{u},f] \leq 0,
\end{equation}
given the unique minimizing pair $(\tilde{u},\tilde{\mathbb{j}})$ in \eqref{argmin} to conclude the proof of the existence and uniqueness part of Theorem \ref{main-thm}. 

\bigskip

To this end, we introduce a perturbed convex minimization problem, defined for every $u^* \in L^2(\O_{xt}, H^{-1}(\O_v))$ in terms of the functional
\begin{equation}\label{defG}
	\mathcal{G}(u^*):= \inf_{u \in \KK(\psi,0)} \left(J[u+g, u^* + f] -\iint \limits_{ \O_{xt}}\langle u^*(\cdot,x,t)|u(\cdot,x,t) \rangle dx\,dt \right).
\end{equation}
Now, we firstly observe that by definition
\begin{equation*}
	\mathcal{G}(0)=\inf_{u \in \KK(\psi,0)}J[u+g,f] = \inf_{u \in \KK(\psi,g)}J[u,f].
\end{equation*}
Thus, the desired inequality \eqref{negfunc} can be equivalently stated 
in terms of $\mathcal{G}$ as follows:
\begin{equation} \label{gpos}
\mathcal{G}(0) \leq 0.
\end{equation} 
Secondly, by considering \eqref{Jcors}, \eqref{Acors} and \eqref{Gw}, we can rewrite \eqref{defG} as
\begin{equation*}
	\mathcal{G}(u^*)= \inf_{(u,\mathbb{j}) \,:\, (u+g, \mathbb{j}) \in \A(f+u^*,\psi,g)}\left( \mathscr{J}[u+g,\mathbb{j}]-\iint \limits_{ \O_{xt}}\langle 
	u^*(\cdot,x,t)|u(\cdot,x,t) \rangle dx\,\right),
\end{equation*}
and observe that $\mathcal{G}$ is a convex, locally bounded from above and lower semi-continuous functional on $L^2(\O_{xt}, H^{-1}(\O_v))$
(see \cite[Lemma 3.4]{LitsgardNystrom}).
Then for every $h \in \left( L^2(\O_{xt}, H^{-1}(\O_v)) \right)^* = L^2(\O_{xt}, H^{1}_c (\O_v) )$ 
we introduce its convex dual $\mathcal{G}^*$, that is defined as
\begin{align*} 
	\mathcal{G}^* (h) &: = \sup \limits_{u^* \in L^2(\O_{xt}, H^{-1}(\O_v)) } \left( - \mathcal{G}(u^*) 
						+ \iint \limits_{ \O_{xt}}\langle u^*(\cdot,x,t)|h(\cdot,x,t) \rangle dx\,dt \right)\\
&=\sup \limits_{(u,\mathbb{j},u^*)} \left\{\,\, -\iiint \limits_{\O_v \times \O_{xt}}
			\frac12 \left( A \left( \nabla_v (u+g) -\mathbb{j} \right) \right)\cdot \left( \nabla_v (u+g) - \mathbb{j} \right)
			\, \dd v \, \dd x \, \dd t\right.\\
			&\qquad \qquad+\left.\iint \limits_{ \O_{xt}}\langle u^*(\cdot,x,t)|h(\cdot,x,t)+u(\cdot,x,t) \rangle dx\,dt  \right\},
\end{align*}
where $(u,\mathbb{j},u^*)\in \W_0 \times (L^2(\O_{xt}, L^2(\O_v))^d \times L^2(\O_{xt}, H^{-1}(\O_v))$.
Now, if we denote by $\mathcal{G}^{**}$ the bidual of $\mathcal{G}$, since $\mathcal{G}$ is lower semi-continuous we are in a position to observe that $\mathcal{G}^{**}=\mathcal{G}$,
see \cite[Lemma I.4.1]{Ekeland}. In particular, we have
\begin{align*}
	\mathcal{G}(0) = \mathcal{G}^{**}(0) = \sup \limits_{h \in L^2(\O_{xt}, H^{1}_c(\O_v)) } \left( - \mathcal{G}^{*}(h) \right).
\end{align*}
Thus, we can reduce the problem of verifying \eqref{negfunc}, already transformed in \eqref{gpos}, to prove the convex dual of 
$\mathcal{G}$ is non-negative, i.e. 
\begin{equation*}
	\mathcal{G}^*(h) \ge 0 \qquad \text{for every $h \in L^2(\O_{xt}, H^1_c(\O_v)$}.
\end{equation*} 
The proof of this fact directly follows as in \cite[Lemma 3.5, 3.6 and 3.7]{LitsgardNystrom}, and therefore is omitted.

To conclude the proof of Theorem \ref{main-thm}, we are left with proving the quantitative estimate \eqref{quantest} for a weak solution $u \in \W(\O_v \times \O_{xt})$ to \eqref{obstacle}. To this end, observing that $u-g$ is a valid test function in Definition \ref{weak-sol}, we get
\begin{equation}\label{u-gtest}
\begin{split}
			0 = 
			 &-\iiint \limits_{\O_{v} \times \O_{xt}} A(v,x,t) \nabla_v u \cdot \nabla_v (u-g) 
				\, \dd v \, \dd x \, \dd t \\
				&-
			     \iint \limits_{\O_{xt}} \langle f (\cdot, x,t) - Y u (\cdot, x,t) | (u-g)(\cdot, x,t) \rangle 
			     \, \dd x \, \dd t  .  
			     \end{split} 
		\end{equation}
		In particular, adding and subtracting the terms 
\begin{align*}
\iiint \limits_{\O_{v} \times \O_{xt}} A(v,x,t) \nabla_v g \cdot \nabla_v (u-g) 
				\, \dd v \, \dd x \, \dd t \quad \textit{\rm{and}}\quad
				 \iint \limits_{\O_{xt}} \langle  Y g (\cdot, x,t) | (u-g)(\cdot, x,t) \rangle 
			     \, \dd x \, \dd t
\end{align*}		
in \eqref{u-gtest}, we obtain
\begin{equation}\label{equality-weaksol}
\begin{split}
0 = 
			 &-\iiint \limits_{\O_{v} \times \O_{xt}} A(v,x,t) \nabla_v (u-g) \cdot \nabla_v (u-g) 
				\, \dd v \, \dd x \, \dd t\\
	&+\boxed{\iint \limits_{\O_{xt}} \langle  Y (u-g) (\cdot, x,t) | (u-g)(\cdot, x,t) \rangle 
			     \, \dd x \, \dd t}_T\\			
	&-\iiint \limits_{\O_{v} \times \O_{xt}} A(v,x,t) \nabla_v g \cdot \nabla_v (u-g) 
				\, \dd v \, \dd x \, \dd t\\
							 &-
			     \iint \limits_{\O_{xt}} \langle f (\cdot, x,t) - Y g (\cdot, x,t) | (u-g)(\cdot, x,t) \rangle 
			     \, \dd x \, \dd t  .   
\end{split}
\end{equation}		
We now focus on the boxed term $T$ in the previous equality. 		
As $u-g \in \W_0$, by definition of $\W_0$, there exists a sequence of functions $h_n \in C_{0,K}^{\infty} (\overline{ \O_{xt} \times \O_v })$ such that
\begin{equation*}
\Vert (u-g)-h_n \Vert_{\W}\to 0, \quad \textit{as } \quad n \to +\infty
\end{equation*}
and in particular
\begin{equation*}
\Vert Y( (u-g)-h_n) \Vert_{L^2(\O_{xt}, H^{-1}(\O_v))}\to 0, \quad \textit{as } \quad n \to +\infty.
\end{equation*}
As a consequence, we obtain 
\begin{align}\label{limsup}
	\iint \limits_{\O_{xt}} &\langle  Y (u-g) (\cdot, x,t) | (u-g)(\cdot, x,t) \rangle \, \dd x \, \dd t\\ \nonumber
			   &\qquad \qquad \leq \limsup_{n \to +\infty}\iint \limits_{\O_{xt}} \langle  Y h_n (\cdot, x,t) | h_n(\cdot, x,t) \rangle 
			     \, \dd x \, \dd t.
\end{align}
As $h_n \in C_{0,K}^{\infty} (\overline{ \O_{xt} \times \O_v })$, we can deal with the last integral in \eqref{limsup} as follows
\begin{equation*} 
\begin{split}
\iint \limits_{\O_{xt}} \langle  Y h_n (\cdot, x,t) | h_n(\cdot, x,t) \rangle 
			     \, \dd x \, \dd t
&=\iiint \limits_{\O_v \times \O_{xt}} Yh_n \, h_n \, dv \, dx\,dt\\
&=\frac{1}{2}\iiint \limits_{\O_v \times \O_{xt}} Yh_n^2 \, dv \, dx\,dt\\
&=\frac{1}{2}\iiint \limits_{\O_v \times \partial\O_{xt}} h_n^2(v,-1)\cdot N_{xt} \, dv \, d\sigma_{x,t} \leq 0,
\end{split}
\end{equation*}
in virtue of the divergence theorem and the definition of the Kolmogorov boundary in \eqref{boundary-k}. 
Hence, we can rewrite equality \eqref{equality-weaksol} as follows
\begin{equation*}
\begin{split}
\iiint \limits_{\O_{v} \times \O_{xt}} A(v,x,t) \nabla_v (u-g) \cdot &\nabla_v (u-g) 
				\, \dd v \, \dd x \, \dd t \\ 
				& \leq			
	-\iiint \limits_{\O_{v} \times \O_{xt}} A(v,x,t) \nabla_v g \cdot \nabla_v (u-g) 
				\, \dd v \, \dd x \, \dd t\\
				&\quad -
			     \iint \limits_{\O_{xt}} \langle f (\cdot, x,t) - Y g (\cdot, x,t) | (u-g)(\cdot, x,t) \rangle 
			     \, \dd x \, \dd t  .   
\end{split}
\end{equation*}		
Thus, taking advantage of \eqref{unifell}, we can rewrite the previous inequality as follows
\begin{equation}\label{ineq-weaksol2}
\begin{split}
\lambda \iint \limits_{\O_{xt}} & \Vert \nabla_v (u-g)(\cdot,x,t)\Vert^2_{L^2(\O_v)} 
				\, \dd x \, \dd t \\
		&\leq			
	\left\vert\,\,\iiint \limits_{\O_{v} \times \O_{xt}} A(v,x,t) \nabla_v g \cdot \nabla_v (u-g) 
				\, \dd v \, \dd x \, \dd t\right\vert\\
		&\qquad + \left\vert\,\, \iint \limits_{\O_{xt}} \langle f (\cdot, x,t) - Y g (\cdot, x,t) | (u-g)(\cdot, x,t) \rangle 
			     \, \dd x \, \dd t \right\vert .   
\end{split}
\end{equation}
We now take care of the right-hand side of \eqref{ineq-weaksol2} employing the boundedness of $A$, Cauchy-Schwarz inequality and Young's inequality and we find
\begin{equation*} 
\begin{split}
&\lambda \iint \limits_{\O_{xt}} \Vert \nabla_v (u-g)(\cdot,x,t)\Vert^2_{L^2(\O_v)}  
				\, \dd x \, \dd t \leq \\
				& \quad\quad\frac{| A|}{2 \epsilon}			\iint \limits_{\O_{xt}} \Vert \nabla_v g(\cdot,x,t)\Vert^2_{L^2(\O_v)} 
				\, \dd x \, \dd t
	+\frac{\epsilon}{2}\iint \limits_{\O_{xt}} \Vert \nabla_v (u-g)(\cdot,x,t)\Vert^2_{L^2(\O_v)} 
				\, \dd x \, \dd t+ \\&\quad\quad\iint \limits_{\O_{xt}} \| f (\cdot, x, t) \|^2_{H^{-1}(\O_v)} \dd x \, \dd t+\iint \limits_{\O_{xt}} \| Y g (\cdot, x, t) \|^2_{H^{-1}(\O_v)} \dd x \, \dd t ,
\end{split}
\end{equation*}
where $\epsilon \in (0,1)$ is a degree of freedom.
Hence, given a suitable choice of $\epsilon$ in the previous inequality, we get
\begin{equation*}
\begin{split}
\Vert\nabla_v (u-g) \Vert_{L^2( \O_{xt}, L^2(\O_{v}))} 
	\leq C \left( \Vert g \Vert_{\W(\O_v \times \O_{xt})}+\Vert f \Vert_{L^2( \O_{xt}, H^{-1}(\O_{v})}\right),
\end{split}
\end{equation*}
where $C=C(d,\epsilon, \lambda)$ is a positive constant.
Then, as $u-g \in L^2( \O_{xt}, H^1_c(\O_{v}))$, we can apply Lemma \ref{lemmapoinc} and infer
\begin{equation} \label{ineq-ug}
\begin{split}
	\Vert u-g \Vert_{L^2( \O_{xt}, H^1(\O_{v}))} \leq C \left( \Vert g \Vert_{\W(\O_v \times \O_{xt})}+\Vert f \Vert_{L^2( \O_{xt}, H^{-1}(\O_{v})}\right).
\end{split}
\end{equation}
We now use the definition of weak solution \eqref{weakformu} against any test function $\phi \in L^2(\O_{xt}, H^1_c(\O_v))$ such that $\Vert \phi \Vert_{L^2(\O_{xt}, H^1_c(\O_v))}= 1$ and, taking advantage once again of the boundedness of $A$, Cauchy-Schwarz inequality and Young's inequality, we easily obtain
\begin{align}\label{ineq-drift}
\Vert Yu \Vert_{L^2( \O_{xt}, H^{-1}(\O_{v})} \leq C \left( \Vert\nabla_v u \Vert_{L^2( \O_{xt}, L^2(\O_{v}))} + \Vert f \Vert_{L^2( \O_{xt}, H^{-1}(\O_{v})} \right).
\end{align}
Combining inequalities \eqref{ineq-ug} and \eqref{ineq-drift}, we prove the quantitative estimate \eqref{quantest} and therefore conclude the proof of Theorem \ref{main-thm}.


\section{A one-sided variational inequality and related open problems}
\label{var-ineq}
As already pointed out in the introduction of this work, our aim is to give rise to the study of 
boundary value problems associated to Kolmogorov operators through variational methods, with a particular attention to the obstacle problem. 

The results we presented in previous sections are only the starting point for this analysis, 
since many interesting open problems are left untouched by our work and 
among which, first and foremost, we recall the proof of a one-on-one correspondence between solutions to \eqref{obstacle} and solutions to a suitable variational inequality.
We observe that the literature concerning evolution equations in the framework of Calculus of Variations is way more recent than the one relevant to elliptic equations, see \cite{BDM, Marcellini} and the references therein.
Moreover, to our knowledge, in our hypoelliptic setting there are not yet available results concerning the relationship between solutions to \eqref{obstacle} and solutions to a suitable variational inequality. 
In the present work, we are able to give proof to 
one of the two implications, i.e. a weak solution to \eqref{obstacle} is also a solution to a suitable variational inequality. 
The other implication is still an open problem. 
	\begin{proposition}
		\label{equivalence2}
		Let us consider $\psi, \varphi \in \W(\O_v \times \O_{xt})$ and the associated obstacle 
		problem \eqref{obstacle} 
		under the assumptions (H) and (D). If $u \in \W_{\psi}(\O_v \times \O_{xt})$
		is a solution to the obstacle problem \eqref{obstacle}, then 
		it satisfies the following variational inequality
		\begin{align*}
				\iint \limits_{\O_{xt}} \langle Y u(\cdot,x,t) | (w - u)(\cdot,x,t) \rangle \dd x \, \dd t +
				&\iiint \limits_{\O_v \times \O_{xt}} A \nabla_v u \cdot \nabla_v \left( w - u \right) \, \dd v \, \dd x \, \dd t\\
				&\ge \iint \limits_{ \O_{xt} } \langle f(\cdot,x,t) | (w - u)(\cdot,x,t) \rangle \dd v \, \dd x \, \dd t,
		\end{align*}
for every $w \in L^2(\O_{xt}, H^1_c(\O_v))$.		
	\end{proposition}
	\begin{proof}
		 Let $u \in \KK(\psi,g) $ be a solution to
		\eqref{obstacle}. Then in particular $u \in \W_{\psi}(\Omega_{v} \times \O_{xt})$ and thus $w-u$ is a valid test function in the weak formulation \eqref{weakformu}, i.e.
		\begin{align*} 
			\iiint \limits_{\O_v \times \O_{xt} }\nabla_v \cdot &\left( A \nabla_v u \right) (w - u)  
			\dd v \, \dd x \, \dd t 
			+ \iint \limits_{\O_{xt} } \langle Y u | w - u \rangle  
			\dd x \, \dd t 
			=   \iint \limits_{ \O_{xt} } \langle f | w - u \rangle \dd v \, \dd x \, \dd t .
		\end{align*}
		By the Divergence Theorem in $H^1(\O_v \times \O_{xt})$, we get
		\begin{align*}
			\iiint \limits_{\O_v \times \O_{xt} } \nabla_v \cdot &\left( A \nabla_v u \right) (w - u)  
			\dd v \, \dd x \, \dd t \\
			&= - \iiint \limits_{\O_v \times \O_{xt} }  A \nabla_v u \cdot
			\nabla_v (w - u)  
			\dd v \, \dd x \, \dd t  
			 + \iint \limits_{\O_{xt}}
			\int \limits_{ \p_{ v} (\O_v) } \gamma_0 \left( A \nabla_v u \right) 
			\cdot {N_v} \, 
			(w - u)  
			\, \dd \sigma \\
			&\ge - \iiint \limits_{\O_v \times \O_{xt} }  A \nabla_v u \cdot
			\nabla_v (w - u)  
			\dd v \, \dd x \, \dd t 
		\end{align*}
		where $N_v$ denotes the outer normal with respect to the velocity boundary, 
		$\gamma_0$ denotes the trace of the function.
	\end{proof}	
The lack of the equivalence between the obstacle problem and a suitable variational inequality 
implies we are not allowed to employ tools from Calculus of Variations that are specifically designed for variational inequalities. 
This means that, for instance, the coincidence set of the obstacle problem is not yet characterized,
and even the equivalence between the two formulations \eqref{obstaclePP} and \eqref{obstacle} of it is still missing. 
In fact, the latter is assumed to be true in the existing literature, see f.i. \cite{obstacle-3}.

Lastly, it is our opinion that an extension of these studies to the more general class of ultraparabolic operators of Kolmogorov type \eqref{ultra} with rough coefficients
would be of great interest for the community. In order to achieve this, one has to overcome many difficulties, starting with the characterization
of the set $\W$ introduced in \cite{AR-harnack} in this more general setting and the definition of the trace of a function in this framework.

\end{document}